\documentclass{amsart}
\usepackage{latexsym,amsmath,amssymb,amscd}

\newtheorem{theorem}{Theorem}

\newtheorem{proposition}[theorem]{Proposition}

\theoremstyle{definition}

\theoremstyle{remark}

\numberwithin{equation}{section}

\begin{document}

\def\sax{\sigma^x_a}

\def\suma{\sum_a}

\def\sumb{\sum_b}

\def\teax{\tilde{E}^x_a}

\def\tfby{\tilde{F}^y_b}

\def\eax{{E}^x_a}

\def\fby{{F}^y_b}

\def\eaxh{({E}^x_a)^{1/2}}

\def\fil{\Phi}

\def\fis{\Phi^*}

\def\bh{B(H)}

\def\c1{C_1(H)}

\def\eaxp{\{E^x_a\}\sb{a,x}}

\def\fbyp{\{F^y_b\}\sb{b,y}}

\def\teaxp{\{\tilde{E}^x_a\}\sb{a,x}}

\def\tfbyp{\{\tilde{F}^y_b\}\sb{b,y}}

\def\pabxy{p(a,b|x,y)}

\title[Tsirelson's problem]{Tsirelson's problem\\
 and purely atomic\\
  von Neumann algebras}
\author{Bebe Prunaru}

\address{Institute of Mathematics ``Simion Stoilow'' 
of the Romanian
Academy\\
P.O. Box 1-764
 RO-014700 Bucharest 
Romania}
\email{bebe.prunaru@gmail.com}

\begin{abstract}

It is shown that if a bipartite behavior admits a field representation in which 
Alice (or Bob's) observable algebra generates a purely atomic von Neumann algebra 
then it is non-relativistic.

\end{abstract}

\maketitle

Let $H$ be a separable complex Hilbert space, and let 
$\bh$ be the algebra of all bounded linear operators on $H$.
If $S\subset\bh$ then $span(S)$ denotes its linear span
 and $comm(S)$ its commutant.

Let $\c1$ be the space of all trace-class operators on $H$.
We denote 

$$<\mu,T>=tr(\mu T) \quad  \mu\in\c1, T\in\bh.$$

If $\Psi:\bh\to\bh$ is a weak star continuous  map, then
we shall denote by $\Psi^*:\c1\to\c1$ its predual map, hence
$$<\Psi^*(\mu),T>=<\mu,\Psi(T)> \quad \mu\in\c1, T\in\bh.$$

In what follows $A$ and $B$ are finite  sets. 
Moreover $\{A_x\}\sb{x\in A}$ and $\{B_y\}\sb{y\in B}$ are families of finite sets.
Elements of $A\sb x$ are identified by  pairs of the form $(a,x)$ and similarly for $B\sb y$.

The following result has been recently proved in \cite{NCPV}.

\begin{theorem} \label{quansal}

 Let $\eaxp$ and $\fbyp$ be two families of positive operators
 in $\bh$ 
 such that

 \begin{itemize}
  
 \item[(i)] $\suma\eax=1 \quad (\forall) x\in A$
 
 \item[(ii)] $\sumb\fby=1 \quad (\forall) y\in B$

\item[(iii)] $\eax\fby=\fby\eax \quad (\forall) a,b,x,y.$

\end{itemize}

Let $\rho\in\c1$ be positive  with $tr(\rho)=1$
and let 
 $$\pabxy=<\rho,\eax\fby>.$$

Suppose there exist a family $\{\sax\}\sb{a,x}$ of positive trace-class operators pn $H$ such that 

\begin{itemize} 

\item[(iv)] $\sigma=\suma\sax$ does not depend on $x\in A$ and

\item[(v)] $<\sax,\fby>=\pabxy$
for all $a,b,x,y$.
\end{itemize}

Then there exist families 
$\teaxp$ and $\tfbyp$ 
of positive operators on $H$ 
 and a normal state $\tilde{\rho}$ on $B(H\otimes H)$
 such that

 $$\suma\teax=1 \quad (\forall) x\in A$$
 and
 $$\sumb\tfby=1 \quad (\forall)y\in B$$ 
  and such that 
$$\pabxy=<\tilde{\rho},\teax\otimes\tfby>\quad (\forall) a,b,x,y.$$

\end{theorem}

This result is related to a certain problem in the theory of quantum correlations 
formulated in \cite{T1} and \cite{T2}. For recent work in this area we refer to \cite{F},
\cite{J},\cite{NCPV},\cite{SW} and the references therein. 
In this paper we shall provide a class of examples where this theorem applies.

\begin{proposition} \label{main}

Suppose $\eaxp$ and $\fbyp$ are families of positive operators on $H$ 
satisfying (i)-(iii) in Theorem \ref{quansal} and let 
 $\rho\in\c1$ be positive  with $tr(\rho)=1$.

Assume there exists a  positive linear  weak star continuous idempotent map 

$$\fil:\bh\to\bh$$  such that

$$span(\fbyp)\subset range(\Phi)\subset comm(\eaxp)$$

Then there exist  a family $\{\sax\}\sb{a,x}$ of positive trace-class operators
 on $H$ such that (iv) and (v) in Theorem \ref{quansal} hold true.

\end{proposition}

\begin{proof}

 Let us define 
$$\sax=\fis(\eaxh\rho\eaxh)$$
for all $a,x$. 
Then for every $T\in\bh$ we have
$$<\suma\sax,T>=<\rho,\suma\eaxh\fil(T)\eaxh>
 =<\rho,\fil(T)>=<\fis(\rho),T>$$
therefore $\sigma=\suma\sax$
 does not depend  on $x$.
Moreover for every $a,x,b,y$ we have 
$$<\sax, \fby>=<\rho, \eaxh\fil(\fby)\eaxh>=<\rho,\eax\fby>=\pabxy.$$

\end{proof}

This   proof is in part inspired by the proof of Thm 5 in \cite{NCPV}.
Recall that a purely atomic von Neumann algebra is one in which the identity is a sum of minimal projections. Obviously, every finite dimensional von Neumann algebra is purely atomic
as well as $\bh$ itself. It is known  that any such algebra is the range of a weak star continuous completely positive idempotent.  It follows that the above result  applies for instance when either 
$\eaxp$ or $\fbyp$ generate  purely atomic von Neumann algebras.
For terminology and results on this class of algebras we refer to \cite{B}.

 \end{document}